\documentclass[letterpaper, 10 pt, conference]{ieeeconf}  

\IEEEoverridecommandlockouts                              
\usepackage{epsfig} 
\usepackage{mathptmx} 
\usepackage{times} 
\usepackage{amsmath} 
\usepackage{amssymb}  

\usepackage{calc}
\usepackage{pdfsync}

\usepackage[usenames,dvipsnames]{xcolor}

\usepackage{tikz}
\usetikzlibrary{arrows}
\usetikzlibrary{graphs,decorations.pathmorphing,decorations.markings}
\usetikzlibrary{calc}

\pgfmathsetmacro\weight{1/2}
\pgfmathsetmacro\third{1/3}
\pgfmathsetmacro\twothirds{2/3}

\tikzset{degil/.style={
            decoration={markings,
            mark= at position 0.5 with {
                  \node[transform shape] (tempnode) {$/$};
                  }
              },
              postaction={decorate}
}
} 

\usepackage{hyperref}   
\hypersetup{
    colorlinks=true,
    linkcolor=blue,
    filecolor=magenta,      
    urlcolor=blue,
        citecolor=red,
    pdftitle={Sharelatex Example},
    pdfpagemode=FullScreen,
}

\usepackage{enumerate}

\newtheorem{theorem}{Theorem}[section]

\newtheorem{proposition}[theorem]{Proposition}
\newtheorem{corollary}[theorem]{Corollary}

\newtheorem{definition}[theorem]{Definition}
\newtheorem{example}[theorem]{Example}

\newtheorem{remark}[theorem]{Remark}

\newcounter{syscounter}
\newenvironment{sysnum}{\begin{list}{($\Sigma{\arabic{syscounter}}$)}%
{\settowidth{\labelwidth}{($\Sigma4$)}
\settowidth{\leftmargin}{($\Sigma4$)~}%
\usecounter{syscounter}}}
{\end{list}}

\newcommand  \esssup {\mathop{\text{ess} \sup} }

\newcommand \N   {\mathbb{N}}
\newcommand \R   {\mathbb{R}}
\newcommand \C   {\mathbb{C}}

\newcommand \K   {\mathcal{K}}
\newcommand \Kinf{\mathcal{K_\infty}}

\newcommand \KL  {\mathcal{KL}}
\newcommand \LL  {\mathcal{L}}

\newcommand{\Uc}{\ensuremath{\mathcal{U}}}
\newcommand{\Vc}{\ensuremath{\mathcal{V}}}

\newcommand \Iff   {\Leftrightarrow}

\newif\ifAndo

\Andotrue

\title{\LARGE \bf
Input-to-State Stability of Nonlinear Parabolic PDEs with Dirichlet Boundary Disturbances
}

\author{Andrii Mironchenko, Iasson Karafyllis and Miroslav Krstic\\
\thanks{A. Mironchenko is with Faculty of Computer Science and Mathematics, University of Passau,
Innstra\ss e 33, 94032 Passau, Germany. Corresponding author,
{\tt\small andrii.mironchenko@uni-passau.de}.
}
\thanks{Iasson Karafyllis Department of Mathematics, National Technical University of Athens, Greece,
{\tt\small iasonkar@central.ntua.gr}.
}
\thanks{Miroslav Krstic is with Department of Mechanical and Aerospace Engineering, University of California, San Diego, USA, 
{\tt\small krstic@ucsd.edu}.
}
\thanks{
A. Mironchenko has been supported by the DFG grant \href{http://www.fim.uni-passau.de/dynamische-systeme/forschung/input-to-state-stability-and-stabilization-of-distributed-parameter-systems/}{"Input-to-state stability and stabilization of distributed parameter systems"} (Wi1458/13-1).
}
\thanks{
The proofs and further applications of the main results can be found in the full version of the paper, see \cite{MKK17}.
}
}

\begin{document}

\maketitle
\thispagestyle{empty}
\pagestyle{empty}

\begin{abstract}
We introduce a monotonicity-based method for studying input-to-state stability (ISS) of nonlinear parabolic equations with boundary inputs.
We first show that a monotone control system is ISS if and only if it is ISS w.r.t. constant inputs. Then we show by means of classical maximum principles that nonlinear parabolic equations with boundary disturbances are monotone control systems. 

With these two facts, we establish that ISS of the original nonlinear parabolic PDE with constant \textit{boundary disturbances} 
is equivalent to ISS of a closely related nonlinear parabolic PDE with constant \textit{distributed disturbances} and zero boundary condition. The last problem is conceptually much simpler and can be handled by means of various recently developed techniques.
\end{abstract}


\section{Introduction}

The concept of input-to-state stability (ISS), which unified Lyapunov and input-output approaches, plays a foundational role in nonlinear control theory \cite{Son08}. It is central for robust stabilization of nonlinear systems \cite{FrK08,KKK95}, design of nonlinear observers \cite{ArK01}, analysis of large-scale networks \cite{JTP94,DRW07} etc.

The success of ISS methodology for ordinary differential equations and time-delay systems, see e.g. \cite{Tee98, PeJ06, PKJ08, Krs08d}, as well as importance of robust control for distributed parameter systems, motivated a rapid development of the ISS theory for abstract infinite-dimensional systems \cite{JLR08, KaJ11, DaM13b, MiI15b, MiW17b, Mir16} and more specifically of partial differential equations (PDEs) \cite{MaP11, PrM12, MiI15b, AVP16, CDP17, PiO17, TPT17} during the last decade. 
Important results achieved include characterizations of ISS and local ISS for broad classes of nonlinear infinite-dimensional systems \cite{MiW17b, Mir16}, nonlinear small-gain theorems for interconnections of $n\in\N$ infinite-dimensional systems \cite{KaJ11, MiI15b}, applications of ISS Lyapunov theory to analysis and control of various classes of PDE systems \cite{PrM12, AVP16, TPT17, CDP17, PiO17} etc.
However, most of these papers are devoted to PDEs with distributed inputs.
\textit{In this work we study input-to-state stability (ISS) of nonlinear parabolic partial differential equations (PDEs) with boundary disturbances on multidimensional spatial domains.} This question naturally arises in such fundamental problems of PDE control as robust boundary stabilization of PDE systems, design of robust boundary observers, stability analysis of cascades of parabolic and ordinary differential equations (ODEs) etc.

It is well known, that PDEs with boundary disturbances can be viewed as evolution equations in Banach spaces with unbounded input (disturbance) operators. This makes the analysis of such systems much more involved than stability analysis of PDEs with distributed disturbances (which are described by bounded input operators), even in the linear case. 

At the same time, ISS of linear parabolic systems w.r.t. boundary disturbances has been studied in several recent papers using different methodologies \cite{AWP12, KaK16b, KaK17a, JNP18, ZhZ17}.
In \cite{AWP12} the authors have shown that parabolic systems with boundary disturbances are ISS w.r.t. $C^1$ norm of disturbances (known as $D^1$-ISS, see \cite[p. 190]{Son08}) provided the system is uniformly asymptotically stable if disturbances are set to zero. 
In \cite{ZhZ17} linear and nonlinear parabolic systems with Neumann or Robin boundary inputs over one-dimensional spatial domain have been studied by means of Lyapunov methods combined with novel Poincare-like inequalities. As a result several criteria for ISS of such systems w.r.t. $C$-norm of boundary inputs have been obtained. 

In \cite{KaK16b, KaK17a, KaK17b} linear parabolic PDEs with Sturm-Liouville operators over 1-dimensional spatial domain have been treated by using two different methods: (i) the spectral decomposition of the solution, and (ii) the approximation of the solution by means of a  finite-difference scheme. 
This made possible to avoid differentiation of boundary disturbances, and to obtain ISS of classical solutions w.r.t. $L^\infty$ norm of disturbances, as well as in weighted $L^2$ and $L^1$ norms. 
An advantage of these methods is that this strategy can be applied also to other types of linear evolution PDEs. At the same time, for multidimensional spatial domains, the computations can become quite complicated.

In \cite{JNP18, JSZ17} ISS of linear boundary control systems has been approached by means of methods of semigroup and admissibility theory. In particular, using the theory of Orlicz spaces and $H_\infty$ calculus, deep relations between ISS and integral input-to-state stability (iISS) have been achieved. This method can be applied to a large class of linear control systems (in particular, the results on equivalence between ISS and integral ISS have been shown in \cite{JSZ17} not only for parabolic systems, but also for the systems governed by analytic semigroups). At the same time, the lack of proper generalizations of admissibility to general nonlinear systems makes this method hard (if possible) to apply to nonlinear distributed parameter systems.

\textit{In this paper, we develop a novel method for investigation of parabolic PDEs with Dirichlet boundary disturbances.} In contrast to previous results, we do not restrict ourselves to linear equations over 1-dimensional spatial domains. Our results are valid for a class of nonlinear equations over multidimensional bounded domains with a sufficiently smooth boundary. Our method is based on the concept of monotone control systems introduced in \cite{AnS03} and inspired by the theory of monotone dynamical systems pioneered by M. Hirsch in 1980-s in a series of papers, beginning with \cite{Hir82}. For the introduction to this theory, a reader may consult \cite{Smi95}. Monotone control systems appear frequently in applications, see e.g. in chemistry and biology \cite{LAS07, Smi95}.
An early effort to use monotonicity methods to study of ISS of infinite-dimensional systems has been made in \cite{DaM10} to show ISS of a particular class of parabolic systems with distributed inputs and Neumann boundary conditions.

Our strategy is as follows. 

First, we show in Section~\ref{sec:ISS_monotone_control_systems} that \textit{a monotone control system (and in particular a nonlinear parabolic PDE system with boundary disturbances) is ISS if and only if it is ISS over a much smaller class of inputs (e.g. constant inputs)}. This is achieved by proving that, for any given disturbance, there is a larger constant disturbance which leads to the larger deviation from the origin, and hence, in a certain sense, the constant disturbances are the "worst case" ones.

Next, in Section~\ref{ISS_nonlinear_parabolic_PDEs}, using maximum and comparison principles for nonlinear parabolic operators, we show that \textit{under certain regularity assumptions nonlinear parabolic equations with Dirichlet boundary disturbances are monotone control systems}. 
In turn, this nice property helps us to show in Section~\ref{ISS_nonlinear_parabolic_PDEs} 
that \textit{ISS of the original nonlinear parabolic PDE with constant \textit{boundary disturbances} is equivalent to ISS of a closely related nonlinear parabolic PDE with constant \textit{distributed disturbances}}. 
The latter problem has been studied extensively in the last years, and a number of powerful results are available for this class of systems. 
In particular, in \cite{MaP11}, constructions of strict Lyapunov functions for certain nonlinear parabolic systems have been provided. In \cite{MiI15b, DaM13}, ISS and integral ISS small-gain theorems for nonlinear parabolic systems interconnected via spatial domain have been proved which give powerful tools to study stability of large-scale parabolic systems on the basis of knowledge of stability of its components. 
Results in this paper make it possible to use this machinery for analysis of ISS of nonlinear parabolic PDEs with boundary inputs. 


Finally, in Section~\ref{sec:Lin_Parabolic_Systems}, we apply above results to linear parabolic problems, which are of specific interest. The obtained ISS criteria can be applied to the problem of ISS stabilization of linear parabolic systems by means of a PDE backstepping method (see for instance \cite{KrS08, SmK10}) in the presence of actuator disturbances, see \cite{MKK17} for more details. 

Next, we introduce some notation used throughout these notes.
By $\R_+$ we denote the set of nonnegative real numbers. For $z \in \R^n$ the Euclidean norm of $z$ is denoted by $|z|$.
For any open set $G\subset \R^n$ we denote by $\partial G$ the boundary of $G$; by $\overline{G}$ the closure of $G$ and by $\mu(G)$ the Lebesgue measure of $G$.
Also for such $G$ and any $p\in[1,+\infty)$ we denote 
by $L^p(G)$ the space of Lebesgue measurable functions $y$ with $\|y\|_p=\Big(\int_G |y(z)|^p dz\Big)^{1/p}$, 
by $L^\infty(G)$ the space of Lebesgue measurable functions $y$ with $\|y\|_\infty=\esssup_{z\in G}|y(z)|$, 
and by $H^k(G)$ the set of $g\in L^2(G)$, which possess weak derivatives up to the $k$-th order, all of which belong to $L^2(G)$.
$C^k(G)$ consists of $k$ times continuously differentiable functions defined on $G$ and $C_0^1(G)$ consists of a functions from $C^1(G)$ which have a compact support.
$H^{1}_0(G)$ is a closure of $C_0^1(G)$ in the norm of $H^1(G)$. 
If $I \subset \R_+$, then $C^{1,2}(I\times G)$ means the set of functions mapping $I\times G$ to $\R$, which are continuously differentiable w.r.t. the first argument and possess continuous second derivatives w.r.t. the second arguments.

Also we will use the following classes of comparison functions.
%
\begin{equation*}
\begin{array}{ll}
{\K} &:= \left\{\gamma:\R_+\rightarrow\R_+\left|\ \gamma\mbox{ is continuous, strictly} \right. \right. \\
&\phantom{aaaaaaaaaaaaaaaaaaa}\left. \mbox{ increasing and } \gamma(0)=0 \right\}, \\
{\K_{\infty}}&:=\left\{\gamma\in\K\left|\ \gamma\mbox{ is unbounded}\right.\right\},\\
{\LL}&:=\left\{\gamma:\R_+\rightarrow\R_+\left|\ \gamma\mbox{ is continuous and strictly}\right.\right.\\
&\phantom{aaaaaaaaaaaaaaaa} \text{decreasing with } \lim\limits_{t\rightarrow\infty}\gamma(t)=0\},\\
{\KL} &:= \left\{\beta:\R_+\times\R_+\rightarrow\R_+\left|\ \beta \mbox{ is continuous,}\right.\right.\\
&\phantom{aaaaaa}\left.\beta(\cdot,t)\in{\K},\ \beta(r,\cdot)\in {\LL},\ \forall t\geq 0,\ \forall r >0\right\}. \\
\end{array}
\end{equation*}

\textit{Due to the page limits we omit most of the proofs of the results obtained in the article. Please refer to the full version of the paper for the proofs and further applications of the main results \cite{MKK17}.}


\section{Monotonicity of control systems}
\label{sec:Framework_Monotonicity}


We start with a definition of a control system.
\begin{definition}
\label{Steurungssystem}
Consider a triple $\Sigma=(X,\Uc,\phi)$, consisting of 
\begin{enumerate}[(i)]  
 \item A normed linear space $(X,\|\cdot\|_X)$, called the {state space}, endowed with the norm $\|\cdot\|_X$.
 \item A set of  input values $U$, which is a nonempty subset of a certain normed linear space.
%
    %
    %
%
 \item A normed linear space of inputs $\Uc \subset \{f:\R_+ \to U\}$ endowed with the norm $\|\cdot\|_{\Uc}$.
   We assume that $\Uc$ satisfies \textit{the axiom of shift invariance}, which states that for all $u \in \Uc$ and all $\tau\geq0$ the time
shift $u(\cdot + \tau)$ is in $\Uc$.

\item A family of nonempty sets $\{\Uc(x) \subset \Uc: x \in X\}$, where $\Uc(x)$ is a set of admissible inputs for the state $x$.

\item A transition map $\phi:\R_+ \times X \times \Uc \to X$, defined for any $x\in X$ and any $u\in \Uc(x)$ on a certain subset of $\R_+$.
\end{enumerate}
The triple $\Sigma$ is called a (forward-complete) control system, if the following properties hold:
\begin{sysnum}
 \item \textit{Forward-completeness}: for every $x\in X$, $u\in\Uc(x)$ and for all $t \geq 0$ the value 
$\phi(t,x,u) \in X$ is well-defined.
 \item\label{axiom:Identity} \textit{The identity property}: for every $(x,u) \in X \times \Uc(x)$
          it holds that $\phi(0,x,u)=x$.
 \item \label{axiom:Causality}\textit{Causality}: for every $(t,x,u) \in \R_+ \times X \times
          \Uc(x)$, for every $\tilde{u} \in \Uc(x)$, such that $u(s) =
          \tilde{u}(s)$, $s \in [0,t]$ it holds that $\phi(t,x,u) = \phi(t,x,\tilde{u})$.
  \item \label{axiom:Cocycle} \textit{The cocycle property}: for all $t,h \geq 0$, for all
                  $x \in X$, $u \in \Uc(x)$ we have $u(t+\cdot) \in \Uc(\phi(t,x,u))$ and
\[
\phi(h,\phi(t,x,u),u(t+\cdot))=\phi(t+h,x,u).
\]
\end{sysnum}
\end{definition}
In the above definition $\phi(t,x,u)$ denotes the state of a system at the moment $t \in
\R_+$ corresponding to the initial condition $x \in X$ and the input $u \in \Uc(x)$.
A pair $(x,u)\in X\times\Uc(x)$ is referred to as an \textit{admissible pair}.

\begin{remark}\em
\label{rem:Compatibility_Conditions}
For wide classes of systems, in particular for ordinary differential equations, one can assume that $\Uc(x)=\Uc$ for all $x\in X$, that is every input is admissible for any state. On the other hand, the classical solutions of PDEs with Dirichlet boundary inputs have to satisfy compatibility conditions (see Section~\ref{ISS_nonlinear_parabolic_PDEs}), and hence for such systems $\Uc(x)\neq \Uc$ for any $x\in X$.
Another class of systems for which one cannot expect that $\Uc(x)=\Uc$ for all $x\in X$ are differential-algebraic equations 
(DAEs), see e.g. \cite{KuM06, KrT15}.
\end{remark}

\begin{definition}
A subset $K \subset X$ of a normed linear space $X$ is called a \textit{positive cone} if $K \cap (-K)=\{0\}$ 
and for all $a \in \R_+$ and all $x,y \in K$ it follows that $ax \in K$; $x+y \in K$.
\end{definition}

\begin{definition}
A normed linear space $X$ together with a cone $K \subset X$ is called an \textit{ordered normed linear space} (see \cite{Kra64}), which we denote $(X,K)$ with an order $\leq$ given by $x \leq y\ \Iff \  y-x \in K$. Analogously $x \geq y\ \Iff \  x-y \in K$.
\end{definition}

\begin{definition}
\index{control system!ordered}
We call a control system $\Sigma=(X,\Uc,\phi)$ ordered, if $X$ and $\Uc$ are ordered normed linear spaces.
\end{definition}

An important for applications subclass of control systems are monotone control systems:
\begin{definition} 
An ordered control system $\Sigma=(X,\Uc,\phi)$ is called monotone, provided for all $t \geq 0$, all $x_1, x_2 \in X$ with $x_1 \leq x_2$ and all $u_1 \in\Uc(x_1), u_2 \in \Uc(x_2)$ with $u_1 \leq u_2$ it holds that 
$\phi(t,x_1,u_1)  \leq \phi(t,x_2,u_2)$.
\end{definition}

To treat situations when the monotonicity w.r.t. initial states is not available, the following definition is useful:
\begin{definition} 
An ordered control system $\Sigma=(X,\Uc,\phi)$ is called monotone w.r.t. inputs, provided for all $t \geq 0$, all $x \in X$ and all $u_1,u_2 \in \Uc(x)$ with $u_1 \leq u_2$ it holds that $ \phi(t,x,u_1)  \leq \phi(t,x,u_2)$.
\end{definition}


\section{Input-to-state stability of monotone control systems}
\label{sec:ISS_monotone_control_systems}

Next we introduce the notion of input-to-state stability, which will be central in this paper.
\begin{definition} 
\label{def:ISS_Uc}
Let $\Sigma=\left(X,\mathcal{U},\phi \right)$ be a control system. Let $\Uc_{c}$ be a subset of $\mathcal{U}$. 
System $\Sigma$ is called \textit{input-to-state stable (ISS) with respect to inputs in $\Uc_{c}$} if there exist functions $\beta \in \KL$, $\gamma \in \K$ such that for every $x \in X$ for which $\Uc(x ) \bigcap \Uc_{c}$ is non-empty, the following estimate holds for all $u \in \Uc(x ) \bigcap \Uc_{c}$ and $t\ge 0$:
\begin{eqnarray}
\| \phi(t,x,u) \|_X \leq \beta(\| x \|_X,t) + \gamma( \|u\|_{\Uc}).
\label{eq:ISS_estimate}
\end{eqnarray}
If $\Sigma$ is ISS w.r.t. inputs from $\Uc$, then $\Sigma$ is called \textit{input-to-state stable (ISS)}.
\end{definition}

For applications the following notion, which is stronger than ISS is of importance:
\begin{definition} 
\label{def:exp-ISS_Uc}
Let $\Sigma=\left(X,\mathcal{U},\phi \right)$ be a control system. Let $\Uc_{c}$ be a subset of $\mathcal{U}$. 
System $\Sigma$ is called \textit{exponentially input-to-state stable (exp-ISS) with respect to inputs in $\Uc_{c}$} if there exist constants $M,a$ and $\gamma\in\Kinf$ such that for every $x \in X$ for which $\Uc(x ) \bigcap \Uc_{c}$ is non-empty, the following estimate holds for all $u \in \Uc(x ) \bigcap \Uc_{c}$ and $t\ge 0$:
\begin{eqnarray}
\| \phi(t,x,u) \|_X \leq Me^{-at}\| x \|_X + \gamma(\|u\|_{\Uc}).
\label{eq:exp-ISS_estimate}
\end{eqnarray}
If $\Sigma$ is exp-ISS w.r.t. inputs from $\Uc$, then $\Sigma$ is called \textit{exponentially input-to-state stable}.
If in addition $\gamma$ can be chosen to be linear, then $\Sigma$ is exp-ISS with a linear gain function.
\end{definition}


We are also interested in the stability properties of control systems in absence of inputs.
\begin{definition}
\label{def:0UGAS_Uc}
A control system $\Sigma=\left(X,\mathcal{U},\phi \right)$ is {\it globally asymptotically
stable at zero uniformly with respect to the state} (0-UGAS), if there
exists a $ \beta \in \KL$, such that for all $x \in X$: $0 \in \Uc(x)$ and for all $
t\geq 0$ it holds that
\begin{equation}
\label{UniStabAbschaetzung}
\left\| \phi(t,x,0) \right\|_{X} \leq  \beta(\left\| x \right\|_{X},t) .
\end{equation}
\end{definition}

It may be hard to verify the ISS estimate \eqref{eq:ISS_estimate} for all admissible pairs of states and inputs.
Therefore a natural question appears: to find a smaller set of admissible states and inputs, so that validity of the ISS estimate \eqref{eq:ISS_estimate} implies ISS of the system for all admissible pairs (possibly with larger $\beta$ and $\gamma$). 
For example, for wide classes of systems it is enough to check ISS estimates on the properly chosen dense subsets of the space of admissible pairs (a so-called "density argument", see e.g. \cite[Lemma 2.2.3]{Mir12}). As Propositions~\ref{prop:ISS_Monotone_Systems}, \ref{prop:ISS_Monotone_states_inputs_Systems} show, for monotone control systems such sets can be much more sparse.

We start with the case, when all $(x,u) \in X \times \Uc$ are admissible pairs.
\begin{proposition}
\label{prop:ISS_Monotone_Systems}
Let $(X,\Uc,\phi)$ be a control system that is monotone with respect to inputs  with $\Uc(x)=\Uc$ for all $x\in X$. Furthermore, let $\Uc_{c}$ be a subset of $\mathcal{U}$ and let the following two conditions hold:
\begin{enumerate}
 \item[(i)] There exists $\rho \in\Kinf$ so that for any $x_-,x,x_+ \in X$ satisfying $x_-\leq x \leq x_+$ it holds that
\begin{eqnarray*}
\|x\|_X \leq \rho(\|x_-\|_X + \|x_+\|_X).
\end{eqnarray*}
 \item[(ii)] There exists $\eta \in \Kinf$ so that for any $u \in \Uc$ there are $u_-,u_+ \in \Uc_{c}$, satisfying $u_-\leq u \leq u_+$ and $\|u_-\|_{\Uc} \leq \eta(\|u\|_{\Uc})$ and $\|u_+\|_{\Uc} \leq \eta(\|u\|_{\Uc})$. 
\end{enumerate}
Then $(X,\Uc,\phi)$ is ISS if and only if $(X,\Uc,\phi)$ is ISS w.r.t. inputs in $\Uc_{c}$.

Moreover, if $\rho$ is linear, then $\Sigma$ is exp-ISS if and only if $\Sigma$ is exp-ISS w.r.t. inputs in $\Uc_{c}$.
If additionally $\Sigma$ is exp-ISS w.r.t. inputs in $\Uc_{c}$ with a linear gain function and $\eta$ is linear, then $\Sigma$ is exp-ISS
with a linear gain function.
\end{proposition}

Here, as for most results in this paper we omit the proofs due to the page limit constraints. \textit{We refer to \cite{MKK17} for the proofs.}
\begin{example}
\label{rem:Results_in_ODE_context}
Consider ODEs of the form
\begin{eqnarray}
\dot{x}=f(x,u),
\label{eq:ODEsys}
\end{eqnarray}
with $X=\R^n$ with an order induced by the cone $\R^n_+$, $\Uc:=L^\infty(\R_+,\R^m)$ with the order induced by the cone $L^\infty(\R_+,\R^m_+)$ and $\Uc(x)=\Uc$ for all $x\in\R^n$. 
Under assumptions that $f$ is Lipschitz continuous w.r.t. the first argument uniformly w.r.t. the second one and that \eqref{eq:ODEsys} is forward complete,
\eqref{eq:ODEsys} defines a control system $(X,\Uc,\phi)$, where $\phi(t,x_0,u)$ is a state of \eqref{eq:ODEsys} at the time $t$ corresponding to $x(0)=x_0$ and input $u$.

Assume that \eqref{eq:ODEsys} is monotone w.r.t. inputs with such $X$ and $\Uc$ and consider $\Uc_{c}:=\{u\in\Uc:\exists k\in\R^m:\ u(t)=k \text{ for a.e. } t\in\R_+\}$. It is easy to verify that assumptions of Proposition~\ref{prop:ISS_Monotone_Systems} are fulfilled, and hence \eqref{eq:ODEsys}
is ISS iff it is ISS w.r.t. the inputs with constant in time controls.

This may simplify analysis of ISS of monotone ODE systems since input-to-state stable ODEs with constant inputs have some specific properties,
see e.g. \cite[pp. 205--206]{Son08}.
\end{example}


As we argued in Remark~\ref{rem:Compatibility_Conditions}, for many systems $\Uc(x) \neq \Uc$ for some $x\in X$. For such systems Proposition~\ref{prop:ISS_Monotone_Systems} is inapplicable, due to the fact that existence of inputs $u_-,u_+ \in \Uc_{c}$ for a given initial condition $x$ and for an input $u$, satisfying assumption (ii) of Proposition~\ref{prop:ISS_Monotone_Systems} does not guarantee that the pairs $(x,u_-)$ and $(x,u_+)$ are admissible, which is needed for the proof of Proposition~\ref{prop:ISS_Monotone_Systems}.
However, if $(X,\Uc,\phi)$ is in addition monotone w.r.t. states, the following holds:
\begin{proposition}
\label{prop:ISS_Monotone_states_inputs_Systems}
Let $\Sigma=\left(X,\mathcal{U},\phi \right)$ be a control system that is monotone with respect to states and inputs. 
Assume that $\Uc_{c}\subset\mathcal{U}$ and let the following two conditions hold:
\begin{itemize}
   \item[(i)]There exists $\rho \in \Kinf $ so that for any $x_{-} ,x,x_{+} \in X$ satisfying $x_{-} \le x\le x_{+} $ it holds that
\begin{equation} \label{GrindEQ__19_} 
\left\| x\right\| _{X} \le \rho \left(\left\| x_{-} \right\| _{X} +\left\| x_{+} \right\| _{X} \right).
\end{equation} 
\item[(ii)] There exist $\eta ,\xi \in \Kinf $ so that for every $x\in X$, $u\in \Uc(x)$ and for every $\varepsilon >0$ there exist $x_{-} ,x_{+} \in X$ and $u_{-} \in \Uc(x_{-} ) \bigcap \Uc_{c}$, $u_{+} \in \Uc(x_{+}) \bigcap \Uc_{c}$, satisfying $x_{-} \le x\le x_{+} $, $u_{-} \le u\le u_{+} $, so that the estimates
\begin{equation} \label{GrindEQ__20_} 
\max \left(\left\| u_{-} \right\| _{\mathcal{U}} ,\left\| u_{+} \right\| _{\mathcal{U}} \right)\le \eta \left(\left\| u\right\| _{\mathcal{U}} +\varepsilon \right),
\end{equation} 
\begin{equation} \label{GrindEQ__21_} 
\max \left(\left\| x_{-} \right\| _{X} ,\left\| x_{+} \right\| _{X} \right)\le \xi \left(\left\| x\right\| _{X} +\left\| u\right\| _{\mathcal{U}} +\varepsilon \right).
\end{equation}
\end{itemize}
hold. Then $\Sigma$ is ISS $\Iff$ $\Sigma$ is ISS w.r.t. inputs in $\Uc_{c}$.

Moreover, if $\rho$ and $\xi$ are linear, then $\Sigma$ is exp-ISS if and only if $\Sigma$ is exp-ISS w.r.t. inputs in $\Uc_{c}$.
If additionally $\Sigma$ is exp-ISS w.r.t. inputs in $\Uc_{c}$ with a linear gain function and $\eta$ is linear, then $\Sigma$ is exp-ISS with a linear gain function.
\end{proposition}

\section{Input-to-state stability of nonlinear parabolic equations}
\label{ISS_nonlinear_parabolic_PDEs}

In this section we apply results from Section~\ref{sec:ISS_monotone_control_systems}
 to nonlinear parabolic equations with boundary inputs.
Our strategy is as follows: first we formulate in Proposition~\ref{prop:New_Comparison_Principle} the comparison principle for nonlinear parabolic operators. Afterwards we use this principle in Theorem~\ref{thm:Monotonicity_Parabolic_Systems} to show that under certain technical conditions the nonlinear parabolic systems with boundary disturbances give rise to monotone control systems. Finally, in Theorems~\ref{thm:ISS_Parabolic_Systems_Characterization}, \ref{thm:exp-ISS_Parabolic_Systems_Characterization} we show our main result that \textit{a nonlinear parabolic system with boundary disturbances is ISS if and only if a related system with distributed disturbances is ISS for constant inputs}. The precise statements of these results follow.


Let $G\subset \mathbb{R} ^{n} $ be an open bounded region, let $T>0$ be a constant and denote $D:=(0,T)\times G$. 
Let $CL(D)$ denote the class of functions $x\in C^{0} \left(\overline{D}\right)\bigcap C^{1,2} (D)$.
Denote for each $t\in[0,T]$ and each $x\in CL(D)$ a function $x[t]:\R_+ \to C^0(\overline{G})$ by $x[t]:=x(t,\cdot)$.

Consider the operator $L$ defined for 	 $x\in CL(D)$ by
\begin{eqnarray}
\label{GrindEQ__1_}
(Lx)(t,z)&:=&\frac{\partial \, x}{\partial \, t} (t,z)-\sum _{i,j=1}^{n}a_{i,j} (z)\frac{\partial ^{2} \, x}{\partial \, z_{i} \, \partial \, z_{j} } (t,z) \nonumber\\
&& \qquad -f\big(z,x(t,z),\nabla x(t,z)\big), 
\end{eqnarray}
where $(t,z)\in D$, $a_{i,j} \in C^{0} (\overline{G})$ for $i,j=1,...,n$ and $f:G\times \mathbb{R} \times \mathbb{R} ^{n} \to \mathbb{R} $ is a continuous function. The operator $L$ is called uniformly parabolic, if there exists a constant $K>0$ so that for all $\xi =(\xi _{1} ,...,\xi _{n} )\in \mathbb{R} ^{n} $ it holds that
\begin{equation} \label{GrindEQ__2_}
\sum _{i,j=1}^{n}a_{i,j} (z)\xi _{i} \xi _{j}  \ge K\left|\xi \right|^{2}  \quad \mbox{for all} \ z\in \overline{G}.
\end{equation}

We need the following proposition, based upon a classical comparison principle from \cite[Theorem 16, p. 52]{Fri83}.
\begin{proposition}
\label{prop:New_Comparison_Principle}
Let $L$ be uniformly parabolic. Assume that for every bounded set $W\subset \mathbb{R} $ there is a constant $k>0$ such that for every $w_1,w_2\in W$, $z\in G$, $\xi \in \mathbb{R} ^{n}$ with $w_1>w_2$ it holds that
\begin{equation} 
\label{GrindEQ__3_} 
f(z,w_1,\xi )-f(z,w_2,\xi ) < k(w_1-w_2).
\end{equation} 
Let $x,y\in CL(D)$ be so that
\begin{eqnarray}
(Ly)(t,z)&\ge& 0\ge (Lx)(t,z), \quad \mbox{for all} \ (t,z)\in D,\label{GrindEQ__4_}\\
y(0,z)&\ge& x(0,z), \quad \mbox{for all} \ z\in G, \label{GrindEQ__5_}\\
y(t,z)&\ge& x(t,z), \quad \mbox{for all} \ (t,z)\in [0,T]\times \partial G. \label{GrindEQ__6_}
\end{eqnarray}
Then $y(t,z)\ge x(t,z)$ for all $(t,z)\in \overline{D}$.
\end{proposition}

Since our intention is to analyze forward complete systems, we introduce some more notation.
Let $CL$ denote the class of functions $x\in C^{0} \left(\mathbb{R} _{+} \times \overline{G}\right)\bigcap C^{1,2} ((0,+\infty )\times G)$.

Now we apply the established results to analyze the initial boundary value problem:
\begin{eqnarray}
(Lx)(t,z)&=&0, \quad \mbox{for all }(t,z)\in (0,+\infty )\times G, \label{GrindEQ__22_}\\
x(0,z)&=&x_{0} (z), \quad \mbox{for all }z\in G, \label{GrindEQ__23_}\\
x(t,z)&=&u(t,z), \quad \mbox{for all }(t,z)\in \mathbb{R} _{+} \times \partial G,  \label{GrindEQ__24_}
\end{eqnarray}
where $x_{0} \in C^{0} (\overline{G})$, $L$ is the uniformly parabolic operator defined by \eqref{GrindEQ__1_} with $f\in C^{0} (\overline{G}\times \mathbb{R} \times \mathbb{R} ^{n} )$.

In this section we assume that the space of input values is $U=C^{0} (\partial G)$, endowed with the standard sup-norm
and that 
$\mathcal{U}:=\left\{\, u\in C^{0} (\mathbb{R} _{+} ;U):\, \, u \mbox{ is bounded}\right\}$, endowed with
\begin{itemize}
  \item the partial order $\le $ for which $u\le v$ iff $u(t,z)\le v(t,z)$ for all $(t,z)\in \mathbb{R} _{+} \times \partial G$,
 \item the norm 
$\left\| u\right\| _{\mathcal{U}} =\mathop{\sup }\limits_{t\ge 0} \left\| u[t]\right\| _{U} =\mathop{\sup }\limits_{z\in \partial G,\ t\ge 0} \left|u(t,z)\right|$. 
\end{itemize}
In the sequel we will need also a subspace of $\Uc$ consisting of constant in time and space inputs:
\begin{eqnarray}
\Uc_{c}:=\{u \in \Uc:\exists k \in\R:\ u(t,z)=k \ \forall (t,z)\in \mathbb{R} _{+} \times \partial G\}.
\label{eq:Uc_const_def}
\end{eqnarray}

Define for $x\in C^0(G)$ the standard $L^p(G)$-norm as $\left\|x \right\| _{p}:=\Big(\int_G |x(z)|^p dz\Big)^{1/p}$.
The Euclidean distance between $w\in G$ and $W \subset \overline{G}$ is denoted by $\rho(w,W):=\inf_{s\in W}|w-s|$.

The following assumptions will be instrumental:
\begin{itemize}
   \item[(H1)] There exists a linear space $X\subseteq C^{0} (\overline{G})$, containing the functions $\{x\in C^{0} (\overline{G}):\exists k\in\R \mbox{ s.t. } x(\cdot)=k\}$, such that for each $x_{0} \in X$ which is constant on $\partial G$, there exists a set of inputs $\Uc(x_{0} )\subseteq \left\{\, v\in \Uc\, :\, v(0,z)=x_{0} (z)\, \mbox{for}\, \, z\in \partial G\, \right\}$, which contains constant in time and space inputs
\[
\left\{\, v\in \Uc\, :\, v(t,z)=x_{0} (z)\, \mbox{for}\, \, z\in \partial G\, \mbox{and}\, \, t\ge 0\, \right\}
\] 
with the following property: for every $x_{0} \in X$, $u\in \Uc(x_{0} )$ there exists a solution $x\in CL$ of the initial boundary 
value problem \eqref{GrindEQ__22_}, \eqref{GrindEQ__23_}, \eqref{GrindEQ__24_} for which $x[t]\in X$ for all $t\ge 0$.  

 \item[(H2)] Assume that for every bounded set $W\subset \mathbb{R} $ there exists $k>0$ such that for every $w_1,w_2\in W$, $(t,z)\in D$, $\xi \in \mathbb{R} ^{n} $ with $w_1>w_2$ inequality \eqref{GrindEQ__3_} holds. Moreover, the function $\overline{G}\times \mathbb{R} \times \mathbb{R} ^{n} \ni (z,w,\xi )\mapsto f(z,w,\xi )\in \mathbb{R} $ is continuously differentiable w.r.t. $w\in \mathbb{R} $ and $\xi \in \mathbb{R} ^{n} $.

 \item[(H3)] For every $\delta >0$, $a\in \mathbb{R} $, $x \in X$ there exists a continuous function $k:\overline{G}\to [0,1]$ with $k(z)=1$ for $z\in \partial G$, $k(z)=0$ for all $z\in G$ with $\rho(z,\partial G)\ge \delta $ and such that $f\in X$ where $f(z)=(1-k(z))x(z)+a\, k(z)$.
\end{itemize}

\begin{remark}\em
Note that (H3) is a condition on the geometry of the boundary of $G$ which is automatically satisfied when $G$ is an open interval in $\R$.
\end{remark}

We equip $X$ in (H1) with the partial order $\le $ for which $x\le y$ iff $x(z)\le y(z)$ for all $z\in \overline{G}$.
For existence theorems, which can be used to verify (H1), we refer to \cite[Chapters 3, 7]{Fri83}.

The next result assures monotonicity of the initial boundary value problem \eqref{GrindEQ__22_}, \eqref{GrindEQ__23_}, \eqref{GrindEQ__24_}.
\begin{theorem}
\label{thm:Monotonicity_Parabolic_Systems}
Suppose that assumptions (H1), (H2) hold and let $p\in[1,+\infty]$. Let us endow the linear space $X$ in (H1) with the standard $L^{p} (G)$-norm, which we denote by $\left\| x\right\| _{p}$.
Then:
\begin{itemize}
   \item[(i)] Initial boundary value problem \eqref{GrindEQ__22_}, \eqref{GrindEQ__23_}, \eqref{GrindEQ__24_} gives rise to
 the monotone control system $\Sigma=\left(X,\mathcal{U},\phi \right)$, where $\phi$ is the solution map of \eqref{GrindEQ__22_}, \eqref{GrindEQ__23_}, \eqref{GrindEQ__24_}.
 \item[(ii)] If additionally (H3) holds, then conditions (i) and (ii) of Proposition~\ref{prop:ISS_Monotone_states_inputs_Systems} hold with $\rho$, $\eta$, $\zeta$ being linear functions and $\Uc_{c}$ given by \eqref{eq:Uc_const_def}.  
\end{itemize}
\end{theorem}

We continue to assume that the axioms (H1) and (H2) hold and that $X$ is as in (H1).

Consider now the following equations: 
\begin{align}
\frac{\partial y}{\partial t}(t,z) - \sum_{i,j=1}^n  & a_{ij}(z)\frac{\partial^2 y}{\partial z_i \partial z_j}(t,z) \nonumber\\
&-  f\Big(z,y(t,z)+v(t,z), \nabla y(t,z)\Big) =0,
\label{eq:Nonlinear_Parabolic_Equation_transformed}
\end{align}
where $t >0$, $z \in G$, together with homogeneous Dirichlet boundary conditions
\begin{eqnarray}
y(t,z) = 0, \quad z \in \partial G,\ t \geq 0.
\label{eq:Dirichlet_transformed}
\end{eqnarray}
The state space of \eqref{eq:Nonlinear_Parabolic_Equation_transformed} is a linear space
\[
Y:=\{y\in X: y(z)=0,\ z\in\partial{G}\}
\]
and the input $v$ belongs to the space of constant in time and space inputs
\[
\Vc:=\{v\in C^0(\R_+\times \overline{G},\R): \exists k\in\R: v(t,z)=k,\ (t,z)\in\R_+\times \overline{G}\}.
\]
Denote the solutions of \eqref{eq:Nonlinear_Parabolic_Equation_transformed}, \eqref{eq:Dirichlet_transformed}, corresponding to the initial condition $y \in Y$ and input $v$ by $\phi_y(t,y,v)$.
It is easy to see that for any $x\in X$ and any $v \in \Vc$ for which $v|_{\partial G} \in \Uc(x)$  it holds that
\begin{eqnarray}
\phi_y(t,y,v)=\phi(t,x,v|_{\partial G}) -v,\quad \mbox{where}\ y=x-v,
\label{eq:phi_phyy_relation}
\end{eqnarray}
and $y \in Y$ since $v\in X$, $X$ is a linear space and $x(z)=v(z)$ for $z\in\partial G$.

Now we are able to state our main result, showing that \textit{ISS of nonlinear parabolic systems w.r.t. a boundary input can be reduced to the problem of ISS of a parabolic system with a distributed and constant input, which is conceptually much simpler.}
\begin{theorem}
\label{thm:ISS_Parabolic_Systems_Characterization}
Suppose that assumptions (H1), (H2) and (H3) hold and let $p\in[1,+\infty]$. Let us endow the linear spaces $X,Y$ with the norm $\left\| \cdot \right\| _{p}$. The following statements are equivalent:

\begin{itemize}
 \item[(i)] The system \eqref{GrindEQ__22_}, \eqref{GrindEQ__23_}, \eqref{GrindEQ__24_}  with the state space $X$ is ISS w.r.t. inputs of class $\Uc$.
 \item[(ii)] The system \eqref{GrindEQ__22_}, \eqref{GrindEQ__23_}, \eqref{GrindEQ__24_} with the state space $X$ is ISS w.r.t. constant in time and space inputs of class $\Uc$.
 \item[(iii)] The system \eqref{eq:Nonlinear_Parabolic_Equation_transformed}, \eqref{eq:Dirichlet_transformed} 
 with the state space $Y$ is ISS w.r.t. inputs in $\Vc$.
\end{itemize}
\end{theorem}

Similarly, the following result on exp-ISS property holds:
\begin{theorem}
\label{thm:exp-ISS_Parabolic_Systems_Characterization}
Suppose that assumptions (H1), (H2) and (H3) hold and let $p\in[1,+\infty]$. Let us endow the linear spaces $X,Y$ with the norm $\left\| \cdot \right\| _{p}$.
The following statements are equivalent:

\begin{itemize}
 \item[(i)] The system \eqref{GrindEQ__22_}, \eqref{GrindEQ__23_}, \eqref{GrindEQ__24_}  with the state space $X$ is exp-ISS w.r.t. inputs of class $\Uc$.
 \item[(ii)] The system \eqref{GrindEQ__22_}, \eqref{GrindEQ__23_}, \eqref{GrindEQ__24_} with the state space $X$ is exp-ISS w.r.t. constant in time inputs of class $\Uc$.
 \item[(iii)] The system \eqref{eq:Nonlinear_Parabolic_Equation_transformed}, \eqref{eq:Dirichlet_transformed} 
 with the state space $Y$ is exp-ISS w.r.t. inputs in $\Vc$.
\end{itemize}
\end{theorem}

%

\section{ISS of linear parabolic systems with boundary inputs}
\label{sec:Lin_Parabolic_Systems}

In this section, we apply the above results to the linear problems with boundary disturbances over multidimensional spatial domains, which are of specific interest.
We continue to assume that $G\subset \R^n$ is an open connected and bounded set with the smooth boundary
and $D:=(0,+\infty)\times G$.

Consider the linear heat equation with a Dirichlet boundary input:
\begin{equation}
\label{eq:IBVP_Operator_L_Linear}
\begin{array}{l}
\frac{\partial x}{\partial t}(t,z) = \Delta x(t,z) + ax(t,z),\quad (t,z) \in D,  \\
x(t,z) = u(t,z), \quad (t,z)\in (0,+\infty) \times \partial G,\\
x(0,z) = x_0(z), \quad z \in \overline{G},
\end{array}
\end{equation}
where $\Delta$ is a Laplacian and $a \in \R$. 

\textit{We are going to show that under reasonable assumptions ISS of the system \eqref{eq:IBVP_Operator_L_Linear}
is equivalent to its uniform global asymptotic stability in absence of inputs.}

The input $u:\R_+ \times\partial G\to \R$ is the trace of a function $\nu\in C^{0}(\R_+\times\overline{G}) \bigcap C^{1,2}((0,+\infty)\times G)$. 
Defining the function
\vspace{-1.5mm}
\begin{eqnarray}
f(t,z):=\Delta \nu(t,z) - \frac{\partial \nu}{\partial t}(t,z),\quad t\geq0,\ z\in G
\label{eq:f_def}
\end{eqnarray}
\vspace{-0.5mm}
and using the transformation
\vspace{-1.5mm}
\begin{eqnarray}
x(t,z)=e^{at}\big(y(t,z)+\nu(t,z)\big),\quad t\geq 0 ,\ z\in \overline{G}
\label{eq:Aux_Transformation}
\end{eqnarray}
\vspace{-0.5mm}
we are in a position to study an equivalent to \eqref{eq:IBVP_Operator_L_Linear} initial boundary value problem
\vspace{-1.5mm}
\begin{equation}
\label{eq:IBVP_Operator_L_Linear_modified}
\begin{array}{l}
\frac{\partial y}{\partial t}(t,z) = \Delta y(t,z) + f(t,z),\quad (t,z) \in D, \\
y(t,z) = 0, \quad (t,z)\in (0,+\infty) \times \partial G,\\
y(0,z) = x(0,z)-\nu(0,z), \quad z \in \overline{G}.
\end{array}
\end{equation}
Let $m>0$ be the smallest integer for which $m\geq \frac{1+[n/2]}{2}$. \cite[Theorem 6, p. 365]{Eva98} in conjunction with \cite[Theorem 4, p. 288]{Eva98}
and \cite[Theorem 6, p. 270]{Eva98} guarantees that if 
\begin{itemize}
 \item[(p-i)] $y[0]\in H^{2m+1}(G)$, $\frac{d^k}{dt^k}(f[t]) \in L^2(0,T;H^{2m-2k}(G))$ for every $k=0,\ldots,m$ and $T>0$,
 \item[(p-ii)] $g_i \in H^1_0(G)$ for $i=0,\ldots,m$, where $g_0:=y[0]$, $g_m:=\frac{d^{m-1}}{dt^{m-1}}f[0]+\Delta g_{m-1}$,
\end{itemize}
then the initial boundary value problem \eqref{eq:IBVP_Operator_L_Linear_modified} has a unique solution $y\in CL$.

Therefore, using the transformation \eqref{eq:Aux_Transformation}, we conclude that for every $x_0\in H^{2m+1}(G)$ and for every input $u\in \R_+\times \partial G \to \R$ being the trace of a function $\nu\in C^{0}(\R_+\times\overline{G}) \bigcap C^{1,2}((0,+\infty)\times G)$ and satisfying (p-i), (p-ii), 
also the initial boundary value problem \eqref{eq:IBVP_Operator_L_Linear} has a unique solution $x\in CL$.
Hence, \eqref{eq:IBVP_Operator_L_Linear} defines a control system with
\begin{itemize}
 \item $X:=H^{2m+1}(G)$ with $\|\cdot\|_p$-norm, for any fixed $p\in[1,+\infty]$.
 \item  $\Uc(x)$ being the set of all inputs $u:\R_{+} \times \partial G\to \R$ which are traces of functions $\nu\in C^{0}(\R_+\times\overline{G}) \bigcap C^{1,2}((0,+\infty)\times G)$ satisfying $\frac{d^{k} }{d\, t^{k} } \left(f[t]\right)\in L^{2} \big(0,T;H^{2m-2k} (G)\big)$ for every $k=0,...,m$ and $T>0$ and $g_{i} \in H_{0}^{1} (G)$ for $i=0,...,m$, where $f$ is defined by 
 \eqref{eq:f_def}, $g_{0} :=x-v[0]$, $g_{1} :=f[0]+\Delta g_{0} $,\dots , $g_{m} :=\frac{d^{m-1} }{d\, t^{m-1} } \left(f[0]\right)+\Delta g_{m-1} $,

 \item $\phi(t,x_0,u)$ being the unique solution $x[t]$ of the initial boundary value problem \eqref{eq:IBVP_Operator_L_Linear}.
\end{itemize}
Notice that $X$ contains the constant functions. We conclude that $\Sigma$ satisfies (H1). Clearly, (H2) holds for $\Sigma$ as well.

Hence we obtain from Theorem~\ref{thm:exp-ISS_Parabolic_Systems_Characterization}:
\begin{corollary}
\label{cor:ISS_Linear_Heat_Equation}
Assume that $G \subset \R^n$ is an open bounded set with a smooth boundary for which Assumption (H3) holds. 
Then $\Sigma=(X,\Uc,\phi)$ is exp-ISS with a linear gain function iff $\Sigma$ is 0-UGAS.
\end{corollary}

\begin{proof}
Clearly, if $\Sigma=(X,\Uc,\phi)$ is exp-ISS, then $\Sigma=(X,\Uc,\phi)$ is 0-UGAS.

Now assume that $\Sigma=(X,\Uc,\phi)$ is 0-UGAS and let us prove the converse implication. 
\eqref{eq:IBVP_Operator_L_Linear} is a problem \eqref{GrindEQ__1_}, corresponding to the operator
$Lx:= \frac{\partial x}{\partial t} - \Delta x - ax$.

According to Theorem~\ref{thm:exp-ISS_Parabolic_Systems_Characterization}, exp-ISS of \eqref{eq:IBVP_Operator_L_Linear} is equivalent to exp-ISS of the system
\begin{eqnarray}
\frac{\partial y}{\partial t} = \Delta y +ay+av,\quad t >0, \ z \in G,
\label{eq:Nonlinear_Parabolic_Equation_transformed-2}
\end{eqnarray}
with homogeneous Dirichlet boundary condition \eqref{eq:Dirichlet_transformed} and constant inputs $v \in \R$.

In order to prove the claim, we are going to use the semigroup approach. Consider three cases: $p\in (1,+\infty)$, $p=1$ and $p=+\infty$.
The operator $A_p$, $p\in (1,+\infty)$ with a domain of definition $D(A_p):=W^{2,p}(G) \bigcap W^{1,p}_0(G)$ and $A_p x:=\Delta x + ax$, for $x\in D(A_p)$ generates an analytic semigroup over $L^p(G)$, see \cite[Theorem 3.6, p. 215]{Paz83}.

For $p=1$ the operator $A_1$ with a domain of definition $D(A_1):= \{x\in W^{1,1}_0(G):\Delta x \in L^1(G)\}$ and $A_1 x:=\Delta x + ax$, for $x\in D(A_1)$ generates an analytic semigroup over $L^1(G)$, see \cite[Theorem 3.10, p. 218]{Paz83}.
Analogously, one can define an operator $A_\infty:=\Delta + aI$ with a certain domain of definition $D(A_\infty)$, which generates an analytic semigroup over $C_0(\overline{G}):=\{x\in C(\overline{G}): x(z)=0 \ \forall z\in\partial G\}$, see \cite[Theorem 3.7, p. 217]{Paz83}.

Now, since we assume that $\Sigma$  is 0-UGAS in the norm $\|\cdot \|_p$, then also the operator $A_p$ generates 0-UGAS (exponentially stable) $C_0$-semigroup $T_p$, which follows since $X$ is dense in the $L^p(G)$ for $p\in[1,\infty)$ and in $C_0(\overline{G})$ for $p=\infty$ and since $T_p$ is a semigroup of bounded operators (hence the norm of the operator $T_p(t)$ with a domain restricted to $Y$ is equal to the norm of $T_p(t)$ as an operator on $X$).


Since $T_p$ is an exponentially stable semigroup, \cite[Proposition 3]{DaM13} ensures that \eqref{eq:Nonlinear_Parabolic_Equation_transformed-2} is exp-ISS with a linear gain function. 
An inspection of Theorem~\ref{thm:exp-ISS_Parabolic_Systems_Characterization} shows that \eqref{eq:IBVP_Operator_L_Linear} is exp-ISS with a linear gain function as well.
%
%
%
\end{proof}

\begin{remark}\em
\label{rem:Linfinity}
Note that the operator $\Delta + aI$ in the above proof does not generate a strongly continuous semigroup over the space $L^\infty(G)$, see \cite[p.217]{Paz83} or \cite[Lemma 2.6.5, Remark 2.6.6]{CaH98}, therefore it is of importance to define $A_\infty$ as a generator of strongly continuous semigroup over $C_0(\overline{G})$.
\end{remark}

\begin{remark}\em
\label{rem:Practical_applications}
There are different ways to ensure that $\Sigma$ is 0-UGAS. For $p\in(1,+\infty)$ a usual way would be to construct a Lyapunov functional. 
On the other hand, since $T_p$ is an analytic semigroup for any $p\in[1,+\infty]$,
$T_p$ is exponentially stable (i.e. 0-UGAS) if and only if the spectrum of $A_p$ lies in $\{z\in \C:\text{Re}z<0\}$, see e.g. \cite[p.387]{Tri75}.
\end{remark}

\section{Conclusions}

We presented a new technique for analyzing ISS of linear and nonlinear parabolic equations with boundary inputs.
We prove that parabolic equations with Dirichlet boundary inputs are monotone control systems and use this fact to transform the parabolic system with boundary disturbances into a related system with distributed constant disturbances. We show that ISS of the original equation is equivalent to ISS of the transformed system.
Analysis of ISS of the transformed system is much easier to perform, for example, by means of Lyapunov methods.

Although in this paper we concentrate on parabolic scalar equations and study properties of classical solutions of such equations, the scheme which we have developed here can be useful for other classes of monotone control systems: monotone parabolic systems, ordinary differential equations, ODE-heat cascades, some classes of time-delay systems. We expect that a big part of our analysis can be transferred to study properties of mild solutions of parabolic systems. Finally, some of our results can be used to study ISS of monotone systems of a general nature.

\vspace{-2mm}

%

\end{document}